\documentclass[11pt,twoside]{amsart}
\usepackage{latexsym,amssymb,amsmath}

\numberwithin{equation}{section}
\newtheorem{theorem}{Theorem}[section]
\newtheorem{lemma}[theorem]{Lemma}
\newtheorem{proposition}[theorem]{Proposition}
\newtheorem{corollary}[theorem]{Corollary}

\theoremstyle{definition}
\newtheorem{definition}[theorem]{Definition}

\begin{document}

%%%%%%%%%%%%%%%%%%%%%%%%%%%%%%%%%%%%%%%%%%%%%%%%%%%%%%%%%%%%%%%%%%%%%

\title[The minimum distance of parameterized codes on projective tori]
{The minimum distance of parameterized codes on projective tori} 

\author{Eliseo Sarmiento}
\address{
Departamento de
Matem\'aticas\\
Centro de Investigaci\'on y de Estudios
Avanzados del
IPN\\
Apartado Postal
14--740 \\
07000 Mexico City, D.F.
}
\email{esarmiento@math.cinvestav.mx}

\thanks{The first author was partially supported by CONACyT. The
second author is a member of the Center for Mathematical Analysis,
Geometry and Dynamical Systems. The third author was partially supported by CONACyT 
grant 49251-F and SNI}

\author{Maria Vaz Pinto}
\address{
Departamento de Matem\'atica\\
Instituto Superior Tecnico\\
Universidade T\'ecnica de Lisboa\\ 
Avenida Rovisco Pais, 1\\ 
1049-001 Lisboa, Portugal 
}\email{ vazpinto@math.ist.utl.pt}

\author{Rafael H. Villarreal}
\address{
Departamento de
Matem\'aticas\\
Centro de Investigaci\'on y de Estudios
Avanzados del
IPN\\
Apartado Postal
14--740 \\
07000 Mexico City, D.F.
}
\email{vila@math.cinvestav.mx}
%\urladdr{http://www.math.cinvestav.mx/$\sim$vila/}

%\keywords{Complete intersections,evaluation codes, parameterized codes, 
%binomial ideals, 
%parameters of a code, finite fields, 
%minimum distance, degree, 
%index of regularity, Hilbert function.}
\subjclass[2000]{Primary 13P25; Secondary 14G50, 14G15, 11T71, 94B27, 94B05.} 

\begin{abstract} Let $X$ be a subset of
a projective space, over a finite field 
$K$, which is parameterized by the monomials arising 
from the edges of a clutter.  Let $I(X)$ be the vanishing ideal of $X$. 
It is shown that $I(X)$ is a complete 
intersection if and only if $X$ is a projective torus. In
this case we determine the minimum distance of any parameterized
linear code arising from $X$. 
\end{abstract}

\maketitle

\section{Introduction}

Let $K=\mathbb{F}_q$  be a finite field with $q$ elements and 
let $y^{v_1},\ldots,y^{v_s}$ be a finite set of monomials.  
As usual if $v_i=(v_{i1},\ldots,v_{in})\in\mathbb{N}^n$, 
then we set 
$$
y^{v_i}=y_1^{v_{i1}}\cdots y_n^{v_{in}},\ \ \ \ i=1,\ldots,s,
$$
where $y_1,\ldots,y_n$ are the indeterminates of a ring of 
polynomials with coefficients in $K$. Consider the following set
parameterized  by these monomials 
$$
X:=\{[(x_1^{v_{11}}\cdots x_n^{v_{1n}},\ldots,x_1^{v_{s1}}\cdots
x_n^{v_{sn}})]\in\mathbb{P}^{s-1}	\vert\, x_i\in K^*\mbox{ for all }i\},
$$
where $K^*=K\setminus\{0\}$ and $\mathbb{P}^{s-1}$ is a projective
space over the field $K$. Following \cite{afinetv} we call $X$ an  
{\it algebraic toric set\/} parameterized  by 
$y^{v_1},\ldots,y^{v_s}$. The set $X$ is a multiplicative group under
componentwise multiplication.

Let
$S=K[t_1,\ldots,t_s]=\oplus_{d=0}^\infty S_d$ 
be a polynomial ring 
over the field $K$ with the standard grading, let $[P_1],\ldots,[P_m]$
be the points of $X$, and 
let $f_0(t_1,\ldots,t_s)=t_1^d$. The {\it 
evaluation map\/} 
\begin{equation}\label{ev-map}
{\rm ev}_d\colon S_d=K[t_1,\ldots,t_s]_d\rightarrow K^{|X|},\ \ \ \ \ 
f\mapsto \left(\frac{f(P_1)}{f_0(P_1)},\ldots,\frac{f(P_m)}{f_0(P_m)}\right)
\end{equation}
defines a linear map of
$K$-vector spaces. This map is well defined, i.e., it is independent
of the choice of representatives $P_1,\ldots,P_m$. The image of ${\rm ev}_d$, denoted by $C_X(d)$,
defines a {\it linear code}. Following \cite{algcodes} we call
$C_X(d)$ a {\it parameterized code\/} of
order $d$. As usual by a {\it linear code\/} we mean a linear subspace of
$K^{|X|}$. 

The definition of $C_X(d)$ can be extended to any finite subset
$X\subset\mathbb{P}^{s-1}$ of a projective space over a field $K$.
Indeed if we choose a degree $d\geq 1$, for each $i$ there is $f_i\in S_d$ such
that $f_i(P_i)\neq 0$ and we can define $C_X(d)$ as
the image of the evaluation map given by 
\begin{equation*}%\label{ev-map}
{\rm ev}_d\colon S_d=K[t_1,\ldots,t_s]_d\rightarrow K^{|X|},\ \ \ \ \ 
f\mapsto
\left(\frac{f(P_1)}{f_1(P_1)},\ldots,\frac{f(P_m)}{f_m(P_m)}\right).
\end{equation*}
In this
generality---the resulting linear code---$C_X(d)$ is called an {\it
evaluation code\/} associated to $X$ \cite{gold-little-schenck}. 
It is also called  a 
{\it projective Reed-Muller code\/} over the set $X$
\cite{duursma-renteria-tapia,GRT}.  
Some families of evaluation codes---including several 
variations of Reed-Muller codes---have been studied extensively using
commutative algebra methods (e.g., Hilbert functions, resolutions, Gr\"obner
bases), see \cite{delsarte-goethals-macwilliams,
duursma-renteria-tapia,gold-little-schenck,GR,GRT,
algcodes,renteria-tapia-ca,renteria-tapia-ca2,sorensen}. In this
paper we use these methods to study parameterized codes over 
finite fields. There are some other papers that
have studied evaluation codes from the commutative algebra
perspective \cite{ballico-fontanari,hansen,tohaneanu}.

The {\it dimension\/} and the {\it length\/} of $C_X(d)$ 
are given by $\dim_K C_X(d)$ and $|X|$ respectively. The dimension
and the length 
are two of the {\it basic parameters} of a linear code. A third
basic parameter is the {\it minimum
distance\/} which is given by 
$$\delta_d=\min\{\|v\|
\colon 0\neq v\in C_X(d)\},$$ 
where $\|v\|$ is the number of non-zero
entries of $v$. The basic parameters of $C_X(d)$ are related by the
{\it Singleton bound\/} for the minimum distance:
$$
\delta_d\leq |X|-\dim_KC_X(d)+1.
$$

The parameters of evaluation codes over finite fields have been computed in a number
of cases. If $X=\mathbb{P}^{s-1}$,  
the parameters of $C_X(d)$ are described in
\cite[Theorem~1]{sorensen}. If $X$ is the image of the affine space
$\mathbb{A}^{s-1}$ under the map $\mathbb{A}^{s-1}\rightarrow
\mathbb{P}^{s-1} $, $x\mapsto [(1,x)]$, the parameters 
of $C_X(d)$ are described in
\cite[Theorem~2.6.2]{delsarte-goethals-macwilliams}. Lower bounds for
the minimum distance 
of evaluation codes have been
shown when $X$ is any complete intersection reduced set of points in
a projective space
\cite{ballico-fontanari,gold-little-schenck,hansen}, and when $X$ is a
reduced Gorenstein set of points \cite{tohaneanu}. Upper bounds for
the minimum distance of certain parameterized codes are given in
\cite{algcodes,d-compl}. In this paper 
we examine the case when $X$ is an algebraic toric set parameterized
by $y_1,\ldots,y_s$. 

The contents of this paper are as follows. In
Section~\ref{prelim-invariants-of-I} we introduce the preliminaries 
and explain the well known connection---via Hilbert
functions---between the invariants of the vanishing 
ideal of $X$ and the parameters of $C_X(d)$, all the 
results of this section are well known. In
Section~\ref{minimum-distance-section} we recall a classical and well
known upper bound for the number of roots of a non-zero polynomial in
$S$ (see Lemmas~\ref{giusti-remark} and \ref{vila-vaz-pinto}). Then, we
show upper bounds for the  
number of roots, over an affine torus, for a certain family of
polynomials in $S$ (see Theorem~\ref{maria-vila-bound}). The
main theorem of 
Section~\ref{minimum-distance-section} is a formula for
the minimum distance of $C_X(d)$, 
where 
$$X=\{[(x_1,\ldots,x_s)]\in\mathbb{P}^{s-1}\vert\, x_i\in
K^*\mbox{ for all
}i\}$$ 
is a {\it projective torus\/} in $\mathbb{P}^{s-1}$ (see
Theorem~\ref{maria-vila-hiram-eliseo}). Evaluation codes associated
to a projective torus are called {\it generalized projective 
Reed-Solomon\/} codes \cite{GRH}. If $X$ is a projective torus in
$\mathbb{P}^1$ or $\mathbb{P}^2$, we recover some formulas of
\cite{GRH,algcodes} for the minimum distance of $C_X(d)$ 
(see Proposition~\ref{minimum-distance-p1-p2}).  

Let $X$ be an algebraic toric set parameterized by
$y^{v_1},\ldots,y^{v_s}$.  
The {\it vanishing ideal\/} of $X$, denoted by $I(X)$, is the
ideal of $S$ generated by the homogeneous polynomials of $S$ that
vanish on $X$. The ideal $I(X)$ is called a {\it complete
intersection\/} if it can be generated by $s-1$ homogeneous
polynomials of $S$. In what follows we assume that $v_1,\ldots,v_s$ are 
the characteristic vectors of
the edges of a clutter (a special sort of hypergraph, see
Definition~\ref{clutter-def}). In 
Section~\ref{ci-section} we are able
to classify when $I(X)$ is a complete intersection 
(see Theorem~\ref{ci->canonical-form} and
Corollary~\ref{ci-characterization}). The main
algebraic fact about $I(X)$ that we need for this classification is a
remarkable result of \cite{algcodes} showing that $I(X)$ is a
binomial ideal. 

The complete intersection property of $I(X)$ has also
been studied in \cite{d-compl}, but from a linear algebra
perspective. 
Let $\phi\colon
\mathbb{Z}^n/L\rightarrow\mathbb{Z}^n/L$ be the multiplication map 
$\phi(\overline{a})=(q-1)\overline{a}$, where $L$ is the subgroup
generated by $\{v_i-v_1\}_{i=2}^s$. In \cite{d-compl} it is shown
that if the clutter is uniform, i.e., all its edges have the same
cardinality, and $q\geq 3$, then $I(X)$ is a
complete intersection if and only if $v_1,\ldots,v_s$ are linearly independent
and the map $\phi$ is injective.     

We show an optimal upper bound for the
regularity of $I(X)$ in terms of the regularity of a complete
intersection (see Proposition~\ref{main-regularity-ff}). This shows
that the complete intersection $I(X)$ from clutters have the largest
possible regularity. 

The ideal $I(X)$ is studied in \cite{d-compl} from the
viewpoint of computational commutative algebra. The degree-complexity
and the reduced 
Gr\"obner basis of $I(X)$, with respect to the reverse
lexicographical order, is examined in \cite[Theorem~4.1]{d-compl}. 

For all unexplained 
terminology and additional information  we refer to
\cite{EisStu} (for the theory of binomial ideals),
\cite{AL,Sta1} (for the theory of polynomial ideals and Hilbert
functions), 
\cite{MacWilliams-Sloane,stichtenoth,tsfasman} (for the theory of 
error-correcting codes and algebraic geometric codes), and
\cite{algcodes} (for the theory of parameterized codes). 

\section{Preliminaries: Hilbert functions and the basic parameters of
codes}\label{prelim-invariants-of-I} 

We continue to use the notation and definitions used in the
introduction. In this section we introduce the basic algebraic
invariants of $S/I(X)$, via Hilbert functions, and we recall their
well known connection with the basic 
parameters of parameterized linear codes. Then, we present some of the
results that will be needed later. 

Recall that the {\it projective space\/} of 
dimension $s-1$ over $K$, denoted by 
$\mathbb{P}^{s-1}$, is the quotient space 
$$(K^{s}\setminus\{0\})/\sim $$
where two points $\alpha$, $\beta$ in $K^{s}\setminus\{0\}$ 
are equivalent if $\alpha=\lambda{\beta}$ for some $\lambda\in K^*$. We
denote the  
equivalence class of $\alpha$ by $[\alpha]$. Let
$X\subset\mathbb{P}^{s-1}$ be an algebraic toric set
parameterized by $y^{v_1},\ldots,y^{v_s}$ and let $C_X(d)$ be a
parameterized code of order $d$. The kernel of the
evaluation map ${\rm ev}_d$, defined in 
Eq.~(\ref{ev-map}), is precisely $I(X)_d$ the degree
$d$ piece of $I(X)$. Therefore there is an isomorphism of $K$-vector spaces
$$S_d/I(X)_d\simeq C_X(d).$$

It is well known that two of the basic parameters of $C_X(d)$ can be expressed 
using Hilbert functions of standard graded algebras
\cite{duursma-renteria-tapia,GRT,algcodes,sorensen}, as we
now  explain. Recall that the
{\it Hilbert function\/} of
$S/I(X)$ is given by 
$$H_X(d):=\dim_K\, 
(S/I(X))_d=\dim_K\, 
S_d/I(X)_d=\dim_KC_X(d).$$
The unique polynomial $h_X(t)=\sum_{i=0}^{k-1}c_it^i\in
\mathbb{Z}[t]$ of degree $k-1=\dim(S/I(X))-1$ such that $h_X(d)=H_X(d)$ for
$d\gg 0$ is called the {\it Hilbert polynomial\/} of $S/I(X)$. The
integer $c_{k-1}(k-1)!$, denoted by ${\rm deg}(S/I(X))$, is 
called the {\it degree\/} or  {\it multiplicity} of $S/I(X)$. In our
situation 
$h_X(t)$ is a non-zero constant because $S/I(X)$ has dimension $1$. 
Furthermore $h_X(d)=|X|$ for $d\geq |X|-1$, see \cite[Lecture
13]{harris}. This means that $|X|$ is the {\it degree\/} 
of $S/I(X)$. Thus, $H_X(d)$ and ${\rm deg}(S/I(X))$ are the
dimension and the length of $C_X(d)$ respectively. 

There are algebraic
methods, based on elimination theory and Gr\"obner bases, to compute
the dimension and the length of $C_X(d)$ \cite{algcodes}. This is one
of the reasons that make some of the basic parameters of parameterized codes
more tractable. However, in general, the problem of computing the minimum
distance of a linear code is difficult because it is NP-hard
\cite{vardy-complexity}.  

The {\it index of regularity\/} of $S/I(X)$, denoted by 
${\rm reg}(S/I(X))$, is the least integer $p\geq 0$ such that
$h_X(d)=H_X(d)$ for $d\geq p$. The degree and the regularity index can be
read off the Hilbert series as we now explain. The Hilbert series of
$S/I(X)$ can be 
written as
$$
F_X(t):=\sum_{i=0}^{\infty}H_X(i)t^i=\sum_{i=0}^{\infty}\dim_K(S/I(X))_it^i=
\frac{h_0+h_1t+\cdots+h_rt^r}{1-t},
$$
where $h_0,\ldots,h_r$ are positive integers. Indeed
$h_i=\dim_K(S/(I(X),t_s))_i$ for $0\leq i\leq r$ and
$\dim_K(S/(I(X),t_s))_i=0$ for $i>r$.  
This follows from the fact that $I(X)$ is a Cohen-Macaulay lattice
ideal of height $s-1$ \cite{algcodes}, and by observing that
$\{t_s\}$ is a regular system of 
parameters for $S/I(X)$ (see \cite{Sta1}). The number $r$ is the
regularity index of $S/I(X)$ and $h_0+\cdots+h_r$ is the degree of
$S/I(X)$ (see \cite[Corollary~4.1.12]{monalg}). In our situation, ${\rm
reg}(S/I(X))$ is the Castelnuovo-Mumford regularity of $S/I(X)$
\cite{eisenbud-syzygies}. We will refer to ${\rm reg}(S/I(X))$ as the
{\it regularity\/} of $S/I(X)$.  

For convenience we recall the following result on complete
intersections. 

\begin{proposition}{\rm\cite[Theorem~1, Lemma~1]{GRH}}\label{ci-summary} If 
$\mathbb{T}=\{[(x_1,\ldots,x_s)]\in\mathbb{P}^{s-1}\vert\, x_i\in
K^*\mbox{ for all }i\}$ is a projective torus 
in $\mathbb{P}^{s-1}$, then
\begin{itemize}
\item[(a)] $I(\mathbb{T})=(\{t_i^{q-1}-t_1^{q-1}\}_{i=2}^s)$. 
\item[(b)] $ F_\mathbb{T}(t)=(1-t^{q-1})^{s-1}/(1-t)^s$.  
\item[(c)] ${\rm reg}(S/I(\mathbb{T}))=(s-1)(q-2)$ and ${\rm
deg}(S/I(\mathbb{T}))=(q-1)^{s-1}$.
\end{itemize}
\end{proposition}

When $I(X)$ is a complete intersection, there is a general formula for
the dimension of any projective Reed-Muller code arising from $X$
\cite{duursma-renteria-tapia}. For a projective torus
one can easily find a formula for the dimension as shown below. 

\begin{corollary}{\cite{duursma-renteria-tapia}} If\/ $\mathbb{T}$ is a projective torus 
in $\mathbb{P}^{s-1}$, then the length of $C_\mathbb{T}(d)$ is
$(q-1)^{s-1}$ and its dimension is given by 
$$
\dim_KC_\mathbb{T}(d)=\sum_{j=0}^{\left\lfloor\frac{d}{q-1}\right\rfloor}(-1)^j{s-1\choose
j}{s-1+d-j(q-1)\choose s-1}.
$$
\end{corollary}

\begin{proof} According to Proposition~\ref{ci-summary}, the length of
$C_\mathbb{T}(d)$ is $(q-1)^{s-1}$ and the Hilbert
series of the graded algebra $S/I(\mathbb{T})$ is given by 
$$ 
F_\mathbb{T}(t)=
\sum_{d=0}^{\infty}H_\mathbb{T}(d)t^d=\frac{(1-t^{q-1})^{s-1}}{(1-t)^s}=
\left[\sum_{j=0}^{s-1}(-1)^j\binom{s-1}{j}t^{j(q-1)}\right] 
\left[\sum_{i=0}^\infty\binom{s-1+i}{s-1}t^i\right].
$$
Hence, comparing the coefficients of $t^d$, we get
$$
H_\mathbb{T}(d)=\sum_{i+j(q-1)=d}(-1)^j\binom{s-1}{j}\binom{s-1+i}{s-1}.
$$
Thus making $i=d-j(q-1)$ we obtain the required expression for $\dim_KC_\mathbb{T}(d)$. 
\end{proof}

In 
Section~\ref{minimum-distance-section} we compute the 
minimum distance of $C_\mathbb{T}(d)$, which was an important piece of
information---from the viewpoint of coding theory---missing in the
literature. 

\section{Minimum distance of parameterized
codes}\label{minimum-distance-section}

We continue to use the notation and definitions used in the
introduction. In this section we determine
the minimum distance of $C_X(d)$ when $X$ is a projective torus in $\mathbb{P}^{s-1}$.  

We begin with a well known and classical general upper bound.

\begin{lemma}{\cite[Lemma 3A, p.~147]{schmidt}}\label{giusti-remark} Let $0\neq G=G(t_1,\ldots,t_s)\in
S$ be a polynomial of total degree $d$. Then the number $N$ of zeros
of $G$ in $\mathbb{F}_q^s$ satisfies 
$$
N\leq 
dq^{s-1}.
$$
If $G$ is homogeneous, then the number of its non-trivial zeros is at
most $d(q^{s-1}-1)$.
\end{lemma}

The proof of this lemma, given in the book of W. M. Schmidt
\cite{schmidt}, can be easily adapted to obtain the following
auxiliary result. 

\begin{lemma}\label{vila-vaz-pinto} Let $0\neq G=G(t_1,\ldots,t_s)\in 
S$ be a polynomial of total degree $d$. If 
$$
Z_G:=\{x\in (K^*)^s\, \vert\, G(x)=0\},
$$
then $|Z_G|\leq d(q-1)^{s-1}$.
\end{lemma}

\begin{lemma}\label{feb21-10} Let $d,d',s$ be positive integers such that
$d=k(q-2)+\ell$ and $d'=k'(q-2)+\ell'$ for some integers
$k,k',\ell,\ell'$ satisfying that $k,k'\geq 0$, $1\leq\ell\leq
q-2$ and $1\leq\ell'\leq q-2$. If $d'\leq d$ and $k\leq s-1$, 
then $k'\leq k$ and 
$$
-(q-1)^{s-k'}+\ell'(q-1)^{s-k'-1}\leq -(q-1)^{s-k}+\ell(q-1)^{s-k-1}.
$$
\end{lemma}

\begin{proof} It is not hard to see that $k'\leq k$. It suffices to
prove the equivalent inequality:
$$
q-1-\ell\leq (q-1)^{k-k'}(q-1-\ell').
$$
If $k=k'$, then $\ell\geq \ell'$ and the inequality holds. If $k\geq
k'+1$, then 
$$
q-1-\ell\leq q-1\leq (q-1)(q-1-\ell')\leq (q-1)^{k-k'}(q-1-\ell'),
$$ 
as required. 
\end{proof}

Let $\mathbb{T}^*=(K^*)^s$ be an {\it affine torus\/}. For
$G=G(t_1,\ldots,t_s)\in S$, we denote 
the set of zeros of $G$ in $\mathbb{T}^*$ by $Z_G$. 

\begin{theorem}\label{maria-vila-bound}
Let $G=G(t_1,\ldots,t_s)\in S$ be a polynomial of 
total degree $d\geq 1$ such that $\deg_{t_i}(G)\leq q-2$ for $i=1,\ldots,s$. 
If $d=k(q-2)+\ell$ with $1\leq \ell\leq q-2$ and
$0\leq k\leq s-1$, then 
$$
|Z_G|\leq
(q-1)^{s-k-1}((q-1)^{k+1}-(q-1)+\ell).
$$
\end{theorem}

\begin{proof} 
By induction on $s$. If $s=1$, then $k=0$ and $d=\ell$. Then
$|Z_G|\leq \ell$ because a non-zero polynomial in one variable of degree $d$ has at
most $d$ roots. Assume $s\geq 2$. By Lemma~\ref{vila-vaz-pinto} we may also assume that 
$k\geq 1$. There are $r\geq 0$ distinct elements $\beta_1,\ldots,\beta_r$ in
$K^*$ and $G'\in S$ such that
$$
G=(t_1-\beta_1)^{a_1}\cdots(t_1-\beta_r)^{a_r}G',\ \ \ \ \ a_i\geq 
1\mbox{ for all } i,
$$
and $G'(\beta,t_2,\ldots,t_s)\neq 0$ for any $\beta\in K^*$. Notice
that $r\leq \sum_ia_i\leq q-2$ because the degree of $G$ in $t_1$ is
at most $q-2$. We can
write $K^*=\{\beta_1,\ldots,\beta_{q-1}\}$. Let $d_i'$ be the degree 
of $G'(\beta_i,t_2,\ldots,t_s)$ and let $d'=\max\{d_i'\vert\ r+1\leq
i\leq q-1\}$. If $d'=0$, then $|Z_{G}|=r(q-1)^{s-1}$ and consequently  
$$
r(q-1)^{s-1}\leq (q-2)(q-1)^{s-1}\leq
(q-1)^{s-k-1}((q-1)^{k+1}-(q-1)+\ell). 
$$
The second inequality uses that $k\geq 1$. 
Thus we may assume that $d'>0$ and also that
$\beta_{r+1},\ldots,\beta_m$ are the elements $\beta_i$ of
$\{\beta_{r+1},\ldots,\beta_{q-1}\}$ such that
$G'(\beta_i,t_2,\ldots,t_s)$ has
positive degree. Notice that $d=\sum_{i}a_i+\deg(G')\geq r+d'$. The
polynomial 
$$H:=(t_1-\beta_1)^{a_1}\cdots(t_1-\beta_r)^{a_r}$$ 
has
exactly $r(q-1)^{s-1}$ roots in $(K^{*})^s$. Hence counting the roots
of $G'$ that are not in $Z_H$ we obtain:
\begin{equation}\label{feb22-10}
|Z_G|\leq
r(q-1)^{s-1}+\sum_{i=r+1}^{m}|Z_{G'(\beta_i,t_2,\ldots,t_s)}|.
\end{equation}
For each $r+1\leq i\leq m$, we can write $d_i'=k_i'(q-2)+\ell_i'$, with $1\leq
\ell_i'\leq q-2$. The proof will be divided in three cases. 

Case (I): Assume $\ell>r$ and $k=s-1$. By
\cite[Theorem~1.2]{alon-cn}, the non-zero polynomial
$G'(\beta_i,t_2,\ldots,t_s)$ cannot be the
zero-function on $(K^{*})^{s-1}$ for any $i$ because its degree in
each of the variables $t_2,\ldots,t_s$ is at most $q-2$. A direct argument to show that 
$G'(\beta_i,t_2,\ldots,t_s)$ cannot be the zero-function on
$(K^{*})^{s-1}$ is to notice that if this non-homogeneous polynomial
vanishes on $(K^{*})^{s-1}$, then it must be a polynomial combination
of $t_2^{q-1}-1,\ldots,t_s^{q-1}-1$, a contradiction. Thus, by Eq.~(\ref{feb22-10}), we get
$$
|Z_G|\leq
r(q-1)^{s-1}+(q-1-r)((q-1)^{s-1}-1)\leq(q-1)^s-(q-1)+\ell.
$$

Case (II): Assume $\ell>r$ and $k\leq s-2$. Then $d-r=k(q-2)+(\ell-r)$
with $1\leq \ell-r\leq q-2$. Since $d_i'\leq d-r$ for $i=r+1,\ldots,m$, by
Lemma~\ref{feb21-10}, we get $k_i'\leq k$ for $r+1\leq i\leq m$. 
Then by induction
hypothesis, using Eq.~(\ref{feb22-10}) and Lemma~\ref{feb21-10}, we
obtain:
\begin{eqnarray*}
|Z_G|&\leq&
r(q-1)^{s-1}+\sum_{i=r+1}^{m}
|Z_{G'(\beta_i,t_2,\ldots,t_s)}|\\ 
&\leq& r(q-1)^{s-1}+\sum_{i=r+1}^m
\left[(q-1)^{(s-1)-k_i'-1}((q-1)^{k_i'+1}-(q-1)+\ell_i')\right]\\
&\leq& r(q-1)^{s-1}+(q-1-r)
\left[(q-1)^{(s-1)-k-1}((q-1)^{k+1}-(q-1)+(\ell-r))\right]\\  
&\leq& 
(q-1)^{s-k-1}((q-1)^{k+1}-(q-1)+\ell).
\end{eqnarray*}

Case (III): Assume $\ell\leq r$. Then we can write 
$d-r=k_2(q-2)+\ell_2$ with $k_2=k-1$ and $\ell_2=q-2+\ell-r$. Notice
that $0\leq k_2\leq s-2$ and $1\leq \ell_2\leq q-2$ because 
$k\geq 1$, $r\leq q-2$ and $k\leq s-1$. Since $d_i'\leq d-r$ for $i>r$, by
Lemma~\ref{feb21-10}, we get $k_i'\leq k_2$ for $i=r+1,\ldots,m$. 
Then by induction hypothesis, using Eq.~(\ref{feb22-10}) and
Lemma~\ref{feb21-10}, we
obtain:
\begin{eqnarray*}
|Z_G|&\leq&
r(q-1)^{s-1}+\sum_{i=r+1}^{m}
|Z_{G'(\beta_i,t_2,\ldots,t_s)}|\\ 
&\leq&r(q-1)^{s-1}+\sum_{i=r+1}^m
\left[(q-1)^{(s-1)-k_i'-1}((q-1)^{k_i'+1}-(q-1)+\ell_i')\right]\\
&\leq&r(q-1)^{s-1}+(q-1-r)
\left[(q-1)^{(s-1)-k_2-1}((q-1)^{k_2+1}-(q-1)+\ell_2)\right]\\ 
&=& r(q-1)^{s-1}+(q-1-r)
\left[(q-1)^{s-k-1}((q-1)^{k}-(q-1)+(q-2+\ell-r))\right]\\ 
&\leq& 
(q-1)^{s-k-1}((q-1)^{k+1}-(q-1)+\ell).
 \end{eqnarray*}
The last inequality uses that $r\leq q-2$. This completes the proof of the result.
\end{proof}

We come to the main result of this section.

\begin{theorem}\label{maria-vila-hiram-eliseo} 
If $X=\{[(x_1,\ldots,x_s)]\in\mathbb{P}^{s-1}\vert\,
x_i\in K^*\mbox{ for all }i\}$ is a projective torus in
$\mathbb{P}^{s-1}$ and $d\geq 1$, 
then the minimum distance of $C_X(d)$ is given by 
$$
\delta_d=\left\{\begin{array}{cll}
(q-1)^{s-(k+2)}(q-1-\ell)&\mbox{if}&d\leq (q-2)(s-1)-1,\\
1&\mbox{if}&d\geq (q-2)(s-1),
\end{array}
 \right.
$$
where $k$ and $\ell$ are the unique integers such that $k\geq 0$,
$1\leq \ell\leq q-2$ and $d=k(q-2)+\ell$. 
\end{theorem}

\begin{proof} First we consider the case $1\leq d\leq (q-2)(s-1)-1$.
Then, in this case, we have that $k\leq s-2$. Let $\prec$ be the
graded reverse lexicographical order on 
the monomials of $S$. In this order $t_1\succ\cdots\succ t_s$. Let
$F$ be a homogeneous polynomial of $S$ of degree $d$ such that $F$
does not vanish on all $X$. 
By the division algorithm \cite[Theorem~1.5.9, 
p.~30]{AL}, we can write
\begin{equation}\label{march12-10}
F=h_1(t_1^{q-1}-t_s^{q-1})+\cdots+h_{s-1}(t_{s-1}^{q-1}-t_s^{q-1})+F',
\end{equation}
where $F'$ is a homogeneous polynomial with $\deg_{t_i}(F')\leq q-2$ for $i=1,\ldots,s-1$ and 
$\deg(F')=d$. Let $d'$ be the degree of the polynomial
$F'(t_1,\ldots,t_{s-1},1)$. Consider the sets: 
\begin{eqnarray*}
Z_{F(t_1,\ldots,t_{s-1},1)}&=&\{(x_1,\ldots,x_{s-1},1)\in
(K^*)^{s-1}\times\{1\}\, \vert\,
F(x_1,\ldots,x_{s-1},1)=0\},\\
A_F&=&\{[x]\in X\, \vert\, F(x)=0\}.
\end{eqnarray*}
Notice that there is a bijection
$$
Z_{F(t_1,\ldots,t_{s-1},1)}\stackrel{\psi}{\longrightarrow} A_F,\ \ \ \ \ \ \
(x_1,\ldots,x_{s-1},1)\stackrel{\psi}{\longmapsto} [(x_1,\ldots,x_{s-1},1)].
$$
Indeed $\psi$ is clearly well defined and injective. To see that
$\psi$ is onto take a point $[x]$ in $A_F$ with $x=(x_1,\ldots,x_s)$.
As $F$ is homogeneous of degree $d$, form the equality 
$$
F(x_1/x_s,\ldots,x_{s-1}/x_s,1)=F(x_1,\ldots,x_s)/x_s^d=0,
$$
we get that $p=(x_1/x_s,\ldots,x_{s-1}/x_s,1)$ is a point in
$Z_{F(t_1,\ldots,t_{s-1},1)}$ and $\psi(p)=[x]$. Hence
$|A_F|=|Z_{F(t_1,\ldots,t_{s-1},1)}|$. Using 
Eq.~(\ref{march12-10}), we get
$Z_{F(t_1,\ldots,t_{s-1},1)}=Z_{F'(t_1,\ldots,t_{s-1},1)}$. We set  
$$
H=H(t_1,\ldots,t_{s-1})=F'(t_1,\ldots,t_{s-1},1)\ \mbox{ and }\
Z_H=\{x\in (K^*)^{s-1}\, \vert\, H(x)=0\}.
$$ 
The polynomial $H$
does not vanish on $(K^*)^{s-1}$. This follows from
Eq.~(\ref{march12-10}) and using that $F$ is homogeneous and that $F$
does not vanish on $X$. We may assume that $d'\geq 1$,
otherwise $Z_{F'(t_1,\ldots,t_{s-1},1)}=\emptyset$ and $|A_F|=0$.
Then, 
we can write $d'=k'(q-2)+\ell'$ for some integers $k'\geq 0$ and 
$1\leq\ell'\leq q-2$. Since $k\leq s-2$, by Lemma~\ref{feb21-10}, we
obtain that $k'\leq k$ and
\begin{equation}\label{march11-10}
-(q-1)^{s-1-k'}+\ell'(q-1)^{s-2-k'}\leq
-(q-1)^{s-1-k}+\ell(q-1)^{s-2-k}.
\end{equation}
Then, $k'\leq s-2$ and $H$ is a non-zero polynomial of degree $d'\geq 1$ in $s-1$
variables such that $\deg_{t_i}(H)\leq q-2$ for
$i=1,\ldots,s-1$. Therefore, applying
Theorem~\ref{maria-vila-bound} to $H$ and then using
Eq.~(\ref{march11-10}), we derive
\begin{eqnarray*}
|A_F|=|Z_H|&\leq&
(q-1)^{s-k'-2}((q-1)^{k'+1}-(q-1)+\ell')\\ 
&\leq&(q-1)^{s-k-2}((q-1)^{k+1}-(q-1)+\ell).
 \end{eqnarray*}
Since $F$ was an arbitrary homogeneous polynomial of degree $d$ such that $F$ does
not vanish on $X$ we obtain
$$
M:=\max\{|A_F|\colon F\in S_d;\, F\not\equiv 0\}\leq
(q-1)^{s-k-2}((q-1)^{k+1}-(q-1)+\ell),$$ 
where $F\not\equiv 0$ means that $F$ is not the zero function on $X$. 
We claim that 
$$
M=(q-1)^{s-k-2}((q-1)^{k+1}-(q-1)+\ell).
$$

Let $M_1$ be the expression in the right hand side. It suffices to
show that $M$ is bounded 
from below by $M_1$ or equivalently it suffices to exhibit a homogeneous
polynomial $F\not\equiv 0$ of degree $d$ with exactly $M_1$ roots in
$X$. Let $\beta$ be a generator of
the cyclic group $(K^*,\,\cdot\, )$. Consider the polynomial
$F=f_1f_2\cdots f_kg_\ell$, where
$f_1,\ldots,f_k,g_\ell$ are given by
\begin{eqnarray*}
f_1&=&(\beta t_1-t_2)(\beta^2 t_1-t_2)\cdots (\beta^{q-2} t_1-t_2),\\
f_2&=&(\beta t_1-t_3)(\beta^2 t_1-t_3)\cdots (\beta^{q-2} t_1-t_3),\\
\vdots&\vdots &\ \ \ \ \ \ \ \ \ \ \ \ \ \ \ \ \ \ \vdots\\
f_k&=&(\beta t_1-t_{k+1})(\beta^2 t_1-t_{k+1})\cdots (\beta^{q-2}
t_1-t_{k+1}),\\
g_\ell&=&(\beta t_1-t_{k+2})(\beta^2 t_1-t_{k+2})\cdots (\beta^{\ell}
t_1-t_{k+2}).
\end{eqnarray*}

Now, the roots of $F$ in $X$ are in one to one correspondence with the
union of the sets: 
$$
\begin{array}{c}
\{1\}\times \{ \beta^i \}_{i=1}^{q-2} \times (K^*)^{s-2},\\
\{1\}\times \{1\} \times \{ \beta^i \}_{i=1}^{q-2} \times
(K^*)^{s-3},\\
\vdots\\
\{1\}\times \cdots \times \{1\} \times \{ \beta^i \}_{i=1}^{q-2} \times
(K^*)^{s-(k+1)},\\
\{1\}\times \cdots \times \{1\} \times \{ \beta^i \}_{i=1}^{\ell} \times
(K^*)^{s-(k+2)}.
\end{array}
$$
Therefore the number of zeros of $F$ in $X$ is given by
\begin{eqnarray*}
|A_F|&=&(q-2)(q-1)^{s-2}+(q-2)(q-1)^{s-3}+\cdots+
(q-2)(q-1)^{s-(k+1)}+\ell(q-1)^{s-(k+2)}\\
&=&(q-1)^{s-(k+2)}\left[(q-2)(q-1)^k+\cdots +(q-2)(q-1)+\ell\right]\\
&=&(q-1)^{s-(k+2)}\left[(q-2)(q-1)((q-1)^{k-1}+\cdots
+1)+\ell\right]\\
&=&(q-1)^{s-(k+2)}\left[(q-2)(q-1)
\left(\frac{(q-1)^k-1}{q-2}\right)+\ell\right]\\
&=&(q-1)^{s-(k+2)}\left[(q-1)^{k+1}-(q-1)+\ell\right],
\end{eqnarray*}
as required. Thus $M=M_1$ and the claim is proved. Therefore
\begin{eqnarray*}
\delta_d&=&\min\{\|{\rm ev}_d(F)\|
\colon {\rm ev}_d(F)\neq 0; F\in S_d\}=|X|-\max\{|A_F|\colon F\in
S_d;\, F\not\equiv 0\}\\ 
&=&(q-1)^{s-1}-\left((q-1)^{s-k-2}((q-1)^{k+1}-(q-1)+\ell)\right)\\
&=&(q-1)^{s-k-2}((q-1)-\ell),
\end{eqnarray*}
where $\|{\rm ev}_d(F)\|$ is the number of non-zero
entries of ${\rm ev}_d(F)$. This completes the proof of the case 
$1\leq d\leq (q-2)(s-1)-1$. Next we consider the case $d\geq
(q-2)(s-1)$. By the Singleton bound we readily get that $\delta_d=1$ 
for $d\geq {\rm reg}(S/I(X))$. Hence, applying
Proposition~\ref{ci-summary}, we get $\delta_d=1$ for $d\geq
(s-1)(q-2)$.
\end{proof}

The next proposition is an immediate consequence of our result.
Recall that a linear code is called
{\it maximum distance  separable\/} (MDS for short) if equality holds
in the Singleton bound. 

\begin{proposition}{\rm\cite{GRH,algcodes}}\label{minimum-distance-p1-p2} 
If $X$ is a projective torus in $\mathbb{P}^1$, then $C_X(d)$ is an
MDS code and its minimum distance is given by
$$
\delta_d=\left\{\hspace{-1mm}\begin{array}{cll}
q-1-d&\mbox{if}&1\leq d\leq q-3,\\
1&\mbox{if}&d\geq q-2.
\end{array}
 \right.
$$
If $X$ is a projective torus 
in $\mathbb{P}^2$, then the minimum distance of $C_X(d)$ is given by
$$
\delta_d = \left\{\begin{array}{cll}
(q-1)^2- d(q-1)&\mbox{if}&1\leq d \leq q-2,\\
2q -d-3&\mbox{if}&q-1\leq d \leq 2q-5,\\
1&\mbox{if}& d\geq 2q-4.
\end{array}\right.
$$
\end{proposition}

Parameterized  codes arising from complete
bipartite graphs have been studied in \cite{GR}. In this case one can
use Theorem~\ref{maria-vila-hiram-eliseo} and the next result to  
compute the minimum distance.

\begin{theorem}{\cite{GR}} Let $\mathcal{K}_{k,\ell}$ be a complete
bipartite graph, let $X$ 
be the toric set parameterized by the edges of 
$\mathcal{K}_{k,\ell}$, and let $X_1$ and
$X_2$ be the projective torus of dimension 
$\ell-1$ and $k-1$ respectively. Then, the length, dimension and
minimum distance of $C_X(d)$ are equal to 
$$(q-1)^{k+\ell-2},\ \ H_{X_1}(d)H_{X_2}(d),\ \mbox{ and } \ \ 
\delta_1\delta_2
$$ respectively, where
$\delta_i$ is the minimum distance of $C_{X_i}(d)$. 
\end{theorem}

\section{Complete intersection ideals of parameterized sets of
clutters}\label{ci-section}

We continue to use the notation and definitions used in the
introduction and in the preliminaries. In this section we
characterize the ideals $I(X)$ that are complete intersection when
$X$ arises from a clutter. Then, we show an optimal upper bound for the
regularity of $S/I(X)$. 

\begin{definition}\label{clutter-def}
A {\it clutter\/} $\mathcal{C}$ is a family $E$ of subsets of a
finite ground set $Y=\{y_1,\ldots,y_n\}$ such that if $f_1, f_2 \in
E$, then $f_1\not\subset f_2$. The ground set $Y$ is called the {\em
vertex set} 
of $\mathcal{C}$ and $E$ 
is called the {\em edge set} of $\mathcal{C}$, they are denoted by $V_\mathcal{C}$
and $E_\mathcal{C}$  respectively. 
\end{definition}

Clutters are special hypergraphs. One 
example of a clutter is a graph with the vertices and edges defined in the 
usual way for graphs. 

Let $\mathcal{C}$ be a clutter with vertex set
$V_\mathcal{C}=\{y_1,\ldots,y_n\}$ and let $f$ be an edge of
$\mathcal{C}$. The {\it characteristic vector\/} of $f$ is the vector
$v=\sum_{y_i\in f}e_i$, where
$e_i$ is the $i${\it th} unit vector in $\mathbb{R}^n$. Throughout
this section we assume that $v_1,\ldots,v_s$ is the set of all
characteristic vectors of the edges of $\mathcal{C}$. Recall that the
algebraic toric set parameterized  by 
$y^{v_1},\ldots,y^{v_s}$, denoted by $X$, is the set  
$$
X:=\{[(x_1^{v_{11}}\cdots x_n^{v_{1n}},\ldots,x_1^{v_{s1}}\cdots
x_n^{v_{sn}})]\in\mathbb{P}^{s-1}\vert\, x_i\in K^*\mbox{ for all }i\},
$$
where $v_i=(v_{i1},\ldots,v_{in})\in\mathbb{N}^n$ for $i=1,\ldots,s$.   

\begin{definition} If $a\in {\mathbb R}^s$, its {\it
support\/} is defined as ${\rm supp}(a)=\{i\, |\, a_i\neq 0\}$. 
Note that $a=a^+-a^-$, 
where $a^+$ and $a^-$ are two non-negative vectors 
with disjoint support called the {\it positive\/} and 
{\it negative\/} part of $a$ respectively. 
\end{definition}

\begin{lemma}{\cite[Lemma~3.4]{d-compl}}\label{jan6-10}
Let $\mathcal{C}$ be a clutter. If $f\neq 0$ is a homogeneous
polynomial of $I(X)$ of the form $t_i^b-t^c$ with $b\in\mathbb{N}$,
$c\in\mathbb{N}^s$ and $i\notin{\rm supp}(c)$, 
then $\deg(f)\geq q-1$. Moreover if $b=q-1$, then
$f=t_i^{q-1}-t_j^{q-1}$ for some $j\neq i$.
\end{lemma}

\begin{proof} We may assume that $f=t_1^b-t_2^{c_2}\cdots t_r^{c_r}$, where $c_j\geq 1$ for all $j$ 
and $b=c_2+\cdots+c_r$. Then
\begin{equation}\label{jan7-10}
(x_1^{v_{11}}\cdots x_n^{v_{1n}})^b=(x_1^{v_{21}}\cdots
x_n^{v_{2n}})^{c_2}\cdots (x_1^{v_{r1}}\cdots
x_n^{v_{rn}})^{c_r}\ \mbox{ for all }\ (x_1,\ldots,x_n)\in (K^*)^n.
\end{equation}
We proceed by contradiction. Assume
that $b<q-1$. We claim that if $v_{1k}=1$ for some $1\leq k\leq
n$, then $v_{jk}=1$ for $j=2,\ldots,r$, otherwise if $v_{jk}=0$ for
some $j\geq 2$, then making $x_i=1$ for $i\neq k$ in
Eq.~(\ref{jan7-10}) we get $(x_k^{v_{1k}})^b=x_k^b=x_k^m$, where 
$m<b$. Then $x_k^{b-m}=1$ for $x_k\in K^*$. In  particular if $\beta$
is a generator of the cyclic group $(K^*,\,\cdot\, )$, then
$\beta^{b-m}=1$. Hence $b-m$ is a multiple of $q-1$ and consequently 
$b\geq q-1$, a contradiction. This completes the proof of the claim. 
Therefore ${\rm supp}(v_1)\subset{\rm supp}(v_j)$ for $j=2,\ldots,r$.
Since $\mathcal{C}$ is a clutter we get that $v_1=v_j$ for
$j=2,\ldots,r$, a contradiction because $v_1,\ldots,v_r$
are distinct. Thus $b\geq q-1$. The second part of the lemma follows 
using similar arguments (see \cite{d-compl}).
\end{proof}

A polynomial of the form $f=t^a-t^b$,
with $a,b\in\mathbb{N}^s$, is called a {\it binomial}  
of $S$. The monomials $t^a$ and $t^b$ are called the {\it terms\/} of $f$. An ideal generated 
by binomials is called a {\it binomial ideal\/}.   

\begin{theorem}\label{ci->canonical-form} Let $\mathcal{C}$ be a
clutter. 
If $I(X)$ is a
complete intersection, then
$$
I(X)=(t_1^{q-1}-t_s^{q-1},\ldots,t_{s-1}^{q-1}-t_s^{q-1}).
$$ 
\end{theorem}

\begin{proof} According to \cite[Theorem~2.1]{algcodes} the vanishing ideal $I(X)$
is a binomial ideal. Notice that $I(X)$ has height $s-1$.
Indeed, let $[P]$ be an arbitrary point in $X$, with $P=(\alpha_1,\ldots,\alpha_s)$, and let
$I_{[P]}$ be the ideal generated by the homogeneous polynomials of $S$ that vanish 
at $[P]$. Then 
\begin{equation*}%\label{primdec-ix-1}
I_{[P]}=(\alpha_1t_2-\alpha_2t_1,\alpha_1t_3-\alpha_3t_1,\ldots,
\alpha_1t_s-\alpha_st_1)\ \mbox{ and }\ I(X)=\bigcap_{[P]\in X}I_{[P]}
\end{equation*}
and the later is the primary decomposition of $I(X)$, because $I_{[P]}$ is
a prime ideal of $S$ for any $[P]\in X$. As $I_{[P]}$
has height $s-1$ for any $[P]\in X$, 
we get that the height of $I(X)$ is $s-1$. As $I(X)$ is a complete intersection of height $s-1$, 
there is a minimal set 
$$\mathcal{B}=\{h_1,\ldots,h_{s-1}\}$$  
of homogeneous binomials that generate the ideal $I(X)$. The set
$\mathcal{B}$ is minimal in the sense that
$(\mathcal{B}\setminus\{h_i\})\subsetneq I(X)$ for all $i$. We may assume that
$h_1,\ldots,h_m$ are the binomials of $\mathcal{B}$ that contain a
term of the form $t_i^{c_i}$. By Lemma~\ref{jan6-10} we have that
$\deg(h_i)\geq q-1$ for $i=1,\ldots,m$. Thus we may assume 
that $h_1,\ldots,h_k$ are the binomials of $\mathcal{B}$ of degree
$q-1$ that contain a term of the form $t_i^{q-1}$ and that
$h_{k+1},\ldots,h_m$ have degree greater than $q-1$. By
Lemma~\ref{jan6-10} the binomials $h_1,\ldots,h_k$ have the form
$t_i^{q-1}-t_j^{q-1}$. Notice that $(I(X)\colon t_i)=I(X)$ for
all $i$, this equality follows readily using that $t_i$ does not
vanish at any point of $X$. Hence, by the minimality of $\mathcal{B}$,
the binomials $h_{m+1},\ldots,h_{s-1}$ have both
of their terms not in the set $\{t_1^{a_1},\ldots,t_s^{a_s}\vert\,
a_i\geq 1\mbox{ for all }i\}$. Since $t_i^{q-1}-t_s^{q-1}$ is in $I(X)$
for $i=1,\ldots,s-1$, we can write 
\[
t_i^{q-1}-t_s^{q-1}=\sum_{\ell=1}^k\lambda_{\ell}h_\ell+
\sum_{\ell=k+1}^m\mu_{\ell}h_\ell+
\sum_{\ell=m+1}^{s-1}\theta_{\ell}h_\ell\ \ \ \ \
(\lambda_\ell,\,\mu_\ell,\,\theta_\ell\in S).
\]
As $h_1,\ldots,h_{s-1}$ are homogeneous binomials we can rewrite this
equality as:
\[
t_i^{q-1}-t_s^{q-1}=\sum_{{\ell}=1}^k\lambda_{\ell}'h_{\ell}+
\sum_{{\ell}=m+1}^{s-1}\theta_{\ell}'h_{\ell},
\]
where $\lambda_{\ell}'\in K$ for ${\ell}=1,\ldots,k$ and for each $m+1\leq
{\ell}\leq s-1$ either $\theta_{\ell}'=0$ and $\deg(h_{\ell})>q-1$ 
or $\deg(h_{\ell})\leq q-1$ and
$\deg(h_{\ell})+\deg(\theta_{\ell}')=q-1$. Then 
\[
t_i^{q-1}-t_s^{q-1}-\sum_{{\ell}=1}^k\lambda_{\ell}'h_{\ell}=
\sum_{{\ell}=m+1}^{s-1}\theta_{\ell}'h_{\ell}.
\]
The left hand side of this equality has to be zero, otherwise a
non-zero monomial that occur in the left hand side will have to occur
in the right hand side which is impossible because monomials occurring
on the left have the form ${\lambda}t_j^{q-1}$, $\lambda\in K$, 
and monomials occurring on the
right are never of this form. Hence we get the inclusion
$$
(t_1^{q-1}-t_s^{q-1},\ldots,t_{s-1}^{q-1}-t_s^{q-1})\subset
(h_1,\ldots,h_k).
$$
Since the height of $(h_1,\ldots,h_k)$ is at most $k$, we get $s-1\leq k$.
Consequently $k=s-1$. Thus the inclusion above is an equality as
required.
\end{proof}

\begin{corollary}\label{ci-characterization}
Let $\mathcal{C}$ be a clutter with $s$ edges and let 
$\mathbb{T}=\{[(x_1,\ldots,x_s)]\in\mathbb{P}^{s-1}\vert\, x_i\in K^*\}$ be a
projective torus. The following are equivalent:
\begin{enumerate}
\item[($\mathrm{c}_1$)] $I(X)$ is a complete intersection.
\item[($\mathrm{c}_2$)]
$I(X)=(t_1^{q-1}-t_s^{q-1},\ldots,t_{s-1}^{q-1}-t_s^{q-1})$.
\item[($\mathrm{c}_3$)] $X=\mathbb{T}\subset\mathbb{P}^{s-1}$.
\end{enumerate}
\end{corollary}

\begin{proof} ($\mathrm{c}_1$)$\Rightarrow$($\mathrm{c}_2$): It
follows at once from Theorem \ref{ci->canonical-form}.
($\mathrm{c}_2$)$\Rightarrow$($\mathrm{c}_3$): By
Proposition~\ref{ci-summary} one has 
$I(X)=I(\mathbb{T})=(\{t_i^{q-1}-t_s^{q-1}\}_{i=1}^{s-1})$. As
$X$ and $\mathbb{T}$ are both projective varieties, we get that
$X=\mathbb{T}$ (see \cite[Lemma 4.2]{algcodes} for details). 
($\mathrm{c}_3$)$\Rightarrow$($\mathrm{c}_1$): It
follows at once from Proposition \ref{ci-summary}.
\end{proof}

The next result shows that the regularity of complete intersections
associated to clutters provide an optimal bound for the regularity 
of $S/I(X)$.

\begin{proposition}\label{main-regularity-ff} ${\rm reg}(S/I(X))\leq (q-2)(s-1)$, with equality
if $I(X)$ is a complete intersection associated to a clutter with $s$
edges. 
\end{proposition}

\begin{proof} For $i\geq 0$, we set $h_i=\dim_K(S/(I(X),t_s))_i$. Let
$r$ be the index of regularity of $S/I(X)$. Then, $h_i>0$ 
for $i=0,\ldots,r$ and $h_i=0$ for $i>r$ (see
Section~\ref{prelim-invariants-of-I}). Since $t_s$ does not vanish at
any point of $X$, one has $(I(X)\colon t_s)=I(X)$. Therefore, there is an exact sequence
of graded $S$-modules
$$
0\longrightarrow(S/I(X))[-1]\stackrel{t_s}{\longrightarrow}S/I(X)\longrightarrow
S/(I(X),t_s)\longrightarrow 0,
$$
where $(S/I(X))[-1]$ is the $S$-module with the shifted graduation
such that
$$(S/I(X))[-1]_i=(S/I(X))_{i-1}$$
for all $i$. Therefore from the exact sequence above we get
\begin{eqnarray}
h_i&=&H_X(i)-H_X(i-1)\geq 0\label{jan14-10}
\end{eqnarray}
for $i\geq 1$. On the other hand there is a surjection of graded $S$-modules
$$
D=S/(\{t_i^{q-1}-t_s^{q-1}\}_{i=1}^{s-1}\cup\{t_s\})=
K[t_1,\ldots,t_{s-1}]/(\{t_i^{q-1}\}_{i=1}^{s-1})\longrightarrow
S/(I(X),t_s)\longrightarrow 0.
$$

The Hilbert series of $D$ is equal to the polynomial 
$(1+t+\cdots+t^{q-2})^{s-1}$ because $D$ is a complete intersection 
\cite[p.~104]{monalg}. 
Hence  $D_i=0$ for $i\geq (q-2)(s-1)+1$. From the surjection above we get
that $\dim_K D_i\geq h_i\geq 0$ for all $i$. If $i\geq (q-2)(s-1)+1$,
we obtain  $0=\dim_KD_i\geq h_i\geq 0$. Then, from
Eq.~(\ref{jan14-10}), 
we conclude 
$$
H_X(i)=H_X(i-1)\  \mbox{ for } i-1\geq(q-2)(s-1).
$$
Hence ${\rm reg}(S/I(X))\leq (q-2)(s-1)$. To complete the proof assume
that $I(X)$ is a complete intersection, then by
Corollary~\ref{ci-characterization} the ideal $I(X)$ is equal to 
$(t_1^{q-1}-t_s^{q-1},\ldots,t_{s-1}^{q-1}-t_s^{q-1})$. Consequently 
${\rm reg}(S/I(X))=(q-2)(s-1)$. 
\end{proof}

%An interesting open problem is to classify when $I(X)$ is a complete
%intersection for an arbitrary set $X$ parameterized by monomials. 

Let $X$ be an algebraic toric set parameterized by arbitrary 
monomials $y^{v_1},\ldots,y^{v_s}$. A good parameterized code should
have large $|X|$ and  
with $\dim_KC_X(d)/|X|$ and $\delta_d/|X|$ as large as possible. 
The following easy result gives an indication of where to look for
non-trivial parameterized codes. Only the codes $C_X(d)$ with $1\leq
d<{\rm reg}(S/I(X))$ are interesting. 
\begin{proposition}\label{request-referee} $\delta_d=1$ for
$d\geq{\rm reg}(S/I(X))$.  
\end{proposition}

\begin{proof} Since $H_X(d)$ is equal to the dimension of $C_X(d)$ and
$H_X(d)=|X|$ for $d\geq{\rm reg}(S/I(X))$, by a direct application of
the Singleton  
bound we get that $\delta_d=1$ for $d\geq{\rm reg}(S/I(X))$.
\end{proof}

A well known general fact about parameterized linear codes is that the
dimension of $C_X(d)$ is strictly increasing, as a function of 
$d$, until it reaches a constant value. This behaviour was 
pointed out in \cite{duursma-renteria-tapia} (resp.
\cite{geramita-cayley-bacharach}) 
for finite (resp. infinite)
fields. The minimum distance of $C_X(d)$ has the opposite behaviour
as the following result shows. 

%\cite[Proposition~2.1]{tohaneanu},\cite[Proposition~5.2]{algcodes}
\begin{proposition}{\cite{algcodes,tohaneanu}}\label{minimum-distance-behaviour}
If $\delta_d>1$ 
$($resp. $\delta_d=1)$, then 
$\delta_d>\delta_{d+1}$ $($resp. $\delta_{d+1}=1)$.
\end{proposition}

\medskip

\begin{center}
ACKNOWLEDGMENTS
\end{center}

The authors would like to thank two anonymous referees for providing us 
with useful comments and suggestions, and for pointing out that
Proposition~\ref{minimum-distance-behaviour} was first shown by S.
Toh\v{a}neanu in \cite[Proposition~2.1]{tohaneanu}. The authors would also like thank
Hiram L\'opez, 
who provided an alternative
proof of Theorem~\ref{maria-vila-hiram-eliseo}, and Carlos Renter\'\i a 
for many stimulating discussions.

\bibliographystyle{plain}

\end{document}